\documentclass[dvips,aop,preprint]{imsart}

\usepackage{amsthm,amsmath}
\usepackage{amssymb}
\usepackage[latin9]{inputenc}
\usepackage{enumerate}
\usepackage[pdftex]{graphicx}

\startlocaldefs
\numberwithin{equation}{section}
\newtheorem{thm}{Theorem}[section]
\newtheorem{prop}[thm]{Proposition}
\newtheorem{lem}[thm]{Lemma}
\newtheorem{cor}[thm]{Corollary}
\theoremstyle{definition}
\newtheorem{remark}[thm]{Remark}

\def\journal@name{}
\def\PP{{\mathbb P}}
\def\EE{{\mathbb E}}
\def\UU{{\mathbb U}}
\def\VVar{\mathbf{Var}}

\endlocaldefs

\begin{document}

\begin{frontmatter}

\title{A Lower Bound for the Mixing Time of the Random-to-random Insertions
Shuffle}
\runtitle{A Lower Bound for the Random Insertions
Shuffle}
\thankstext{t1}{Research supported in part by Israel Science Foundation 853/10 and USAFOSR FA8655-11-1-3039.}

\author{\fnms{Eliran} \snm{Subag}\ead[label=e1]{elirans@techunix.technion.ac.il}\thanksref{t1}}
\address{Eliran Subag\\Electrical Engineering\\
 Technion, Haifa, Israel 32000 \\
\printead{e1}
}
\affiliation{Technion}

\runauthor{Eliran Subag}

\begin{abstract}
The best known lower and upper bounds on the total variation mixing time for the random-to-random
insertions shuffle are $\left(\frac{1}{2}-o\left(1\right)\right)n\log n$
and $\left(2+o\left(1\right)\right)n\log n$. A long standing open problem
is to prove that the mixing time exhibits a cutoff. In particular,
 Diaconis conjectured that the cutoff occurs
at $\frac{3}{4}n\log n$. Our main result is a lower bound of
$t_n = \left(\frac{3}{4}-o\left(1\right)\right)n\log n$, corresponding to this conjecture.

Our method is based on analysis of the positions of cards yet-to-be-removed.
We show that for large $n$ and $t_n$ as above, there exists $f(n)=\Theta(\sqrt{n\log n})$ such that,
with high probability, under both the measure induced by the shuffle and the stationary measure, the number of cards within a certain distance from their initial position
is $f(n)$ plus a lower order term.
However, under the induced measure, this lower order term is strongly influenced by
the number of cards yet-to-be-removed, and is of higher order than for
the stationary measure.

\end{abstract}




\end{frontmatter}


\section{Introduction}

In the random-to-random insertions shuffle a card is chosen at random,
removed from the deck and reinserted in a random position. Assuming
the cards are numbered from $1$ to $n$, let us identify an ordered
deck with the permutation $\sigma\in S_{n}$ such that $\sigma\left(j\right)$
is the position of the card numbered $j$. The shuffling process induces
a random walk $\Pi_{t}$, $t=0,1,\ldots$, on $S_{n}$. Let $\PP_{\sigma}^{n}$
be the probability measure corresponding to the random walk starting
from $\sigma\in S_{n}$.

Clearly, $\Pi_{t}$ is an irreducible and aperiodic Markov chain.
Therefore \\ $\PP_{\sigma}^{n}\left(\Pi_{t}\in\cdot\right)$ converges,
as $t\rightarrow\infty$, to the stationary measure $\UU^{n}$, which, since the transition matrix
is symmetric, is the uniform measure on $S_{n}$.
To quantify the distance from stationarity,
one usually uses the total variation (TV) distance
\[
d_{n}\left(t\right)\triangleq\max_{\sigma\in S_{n}}\left\Vert \PP_{\sigma}^{n}\left(\Pi_{t}\in\cdot\right)-\UU^{n}\right\Vert _{TV}=\left\Vert \PP_{id}^{n}\left(\Pi_{t}\in\cdot\right)-\UU^{n}\right\Vert _{TV},
\]
where equality follows since the chain is transitive. The mixing time
is then defined by
\[
t_{mix}^{\left(n\right)}\left(\varepsilon\right)\triangleq\min\left\{ t:\, d_{n}\left(t\right)\leq\varepsilon\right\} .
\]
In order to study the rate of convergence to stationarity for large
$n$, one studies how the mixing time grows as $n\rightarrow\infty$.
In particular, one is interested in finding conditions on $\left(t_{n}\right)_{n=1}^{\infty}$
such that $\lim_{n\rightarrow\infty}d_{n}\left(t_{n}\right)$ equals
$0$ or $1$.

The random-to-random insertions shuffle is known to have a pre-cutoff
of order $O\left(n\log n\right)$. Namely, for $c_{1}=\frac{1}{2}$,
$c_{2}=2$:
\begin{enumerate}[(i)]
\item for any sequence of the form $t_{n}=c_{1}n\log n-k_{n}n$
with $\lim_{n\rightarrow\infty}k_{n}=\infty$, $\lim_{n\rightarrow\infty}d_{n}\left(t_{n}\right)=1$; and
\item for any sequence of the form $t_{n}=c_{2}n\log n+k_{n}n$
with $\lim_{n\rightarrow\infty}k_{n}=\infty$, $\lim_{n\rightarrow\infty}d_{n}\left(t_{n}\right)=0$.
\end{enumerate}
Diaconis and Saloff-Coste \cite{Diaconis93comparisontechniques} showed that the mixing time is
 of order $O\left(n\log n\right)$. Uyemura-Reyes \cite{Uyemura-Reyes} used
a comparison technique from \cite{Diaconis93comparisontechniques}
to show that the upper bound above holds with $c_{2}=4$ and proved
the lower bound with $c_{1}=\frac{1}{2}$ by studying the longest
increasing subsequence. In \cite{Saloff-Coste}  the upper bound is
improved by Saloff-Coste and Z\'u\~niga, also by applying a comparison technique, and shown to hold
with $c_{2}=2$. An alternative proof to the lower bound with $c_{1}=\frac{1}{2}$
is also given there.

A long standing open problem is to prove the existence of a cutoff
in TV (see \cite{Diaconis03mathematicaldevelopments,Diaconis94randomwalks}); that is, a value $c$ such that for any $\varepsilon>0$:
\begin{enumerate}[(i)]
\item for any sequence $t_{n}\leq\left(c-\varepsilon\right)n\log n$, $\lim_{n\rightarrow\infty}d_{n}\left(t_{n}\right)=1$; and
\item for any sequence $t_{n}\geq\left(c+\varepsilon\right)n\log n$,
$\lim_{n\rightarrow\infty}d_{n}\left(t_{n}\right)=0$.
 \end{enumerate}
In particular, in \cite{Diaconis03mathematicaldevelopments} Diaconis
conjectured that there is a cutoff at $\frac{3}{4}n\log n$.

Our main result is a lower bound on the mixing time with this rate.
\begin{thm}
\label{thm:mixing}Let $t_{n}=\frac{3}{4}n\log n-\frac{1}{4}n\log\log n-c_{n}n$ be a sequence of natural numbers
with $\lim_{n\rightarrow\infty}c_{n}=\infty$. Then $\lim_{n\rightarrow\infty}d_{n}\left(t_{n}\right)=1$.
\end{thm}

The proof is based on analysis of the distribution of the positions
of cards yet-to-be-removed. Let $\left[n\right]=\left\{ 1,\ldots,n\right\} $,
and denote the set of cards that have not been chosen for removal
and reinsertion up to time $t$ by $A^{t}=A^{n,t}$. The following
result describes the limiting distribution for a card in $A^{t}$
as the size of the deck grows (in the sense below).

Recall that for a permutation $\sigma \in S^n$, the image of $j$ under $\sigma$, $\sigma(j)$, is the position of card $j$ in the deck with corresponding ordering. Hence, for the random walk $\Pi_t$, $\Pi_t(j)$ corresponds to the position of card $j$ after $t$ random-to-random insertion shuffles.
Let $\Rightarrow$ denote weak convergence and $N\left(0,1\right)$
denote the standard normal distribution.
\begin{thm}
\label{thm:Gaussian delocalization}Let $j_{n}\in\left[n\right]$
and $t_{n}\in\mathbb{N}$ be sequences. Assume that $\gamma\triangleq\lim_{n\rightarrow\infty}\frac{j_{n}}{n}$
exists, and that $$\lim_{n\rightarrow\infty}\frac{n^{2}}{t_{n}j_{n}\left(n-j_{n}\right)}=\lim_{n\rightarrow\infty}\frac{t_{n}}{j_{n}\left(n-j_{n}\right)}=0.$$
 Then
\[
\PP_{{\scriptstyle id}}^{n}\left(\left.\frac{\Pi_{t_{n}}\left(j_{n}\right)-j_{n}}{\sqrt{2t_{n}\lambda_{n}}}\in\cdot\,\right|j_{n}\in A^{t_{n}}\right)\Longrightarrow \PP\left(N\left(0,1\right) \in \cdot \right),
\]
where
\[
\lambda_{n}=\begin{cases}
\frac{j_{n}}{n} & \mbox{if }\gamma=0,\\
\frac{n-j_{n}}{n} & \mbox{if }\gamma=1,\\
\gamma\left(1-\gamma\right) & \mbox{if }\gamma\in\left(0,1\right).
\end{cases}
\]

\end{thm}
This can be explained by the following heuristic. Conditioned on $j\in A^{t}$, $\Pi_m(j)-j$, $m=0,1,\ldots,t$, is
a Markov chain starting at $0$ with increments in $\{0,\pm1\}$. If the increments were independent and identically distributed as the first increment, Theorem \ref{thm:Gaussian delocalization} would have readily followed from Lindeberg's central limit theorem for triangular arrays (\cite{Billingsley95}, Theorem 27.2). While this is not the case, if with high probability the conditional increment distributions (given in \eqref{eq:48} below),
$$\PP_{id}^{n}\left(\left.\Pi_{m+1}\left(j\right)=i+k\right|\Pi_{m}\left(j\right)=i,\, j\in A^{t}\right),\,\,\,\,\,\,k=0,\pm1,$$
are `close enough' to be identical for all the states $\Pi_m(j)$ visits in times $m=0,1,\ldots,t$, one should expect a similar result.
This, however, follows under mild conditions on $t$
and $j$, since the conditional transition probabilities above are very close to being symmetric, and so, with high probability, $\Pi_m(j)$ remains up to time $t$ in a small neighborhood of $j$, where the transition probabilities hardly vary.

To prove the lower bound on the TV distance of $\PP_{id}^n(\Pi_{t_n} \in \cdot)$ and $\UU^n$, we study
the size of sets of the form
\[
\triangle_{\alpha}\left(\sigma\right)\triangleq\left\{ j\in D^{n}:\,\left|\sigma\left(j\right)-j\right|\leq\alpha\sqrt{n\log n}\right\} \,\,,\,\,\sigma\in S_{n},
\]
where $D^{n}=\left[n\right]\cap\left[n\left(1-\varepsilon\right)/2,n\left(1+\varepsilon\right)/2\right]$,
for fixed $\varepsilon\in\left(0,1\right)$ and a parameter $\alpha>0$.
We shall see that for $t_n$ as in Theorem \ref{thm:mixing}, as long as $\limsup c_{n}/\log n <1/4$,  $$\left|\triangle_{\alpha}\right|/\left(2\varepsilon\alpha\sqrt{n\log n}\right) \Longrightarrow 1,$$
 under both measures. However,
the deviation $\left|\triangle_{\alpha}\right|-2\varepsilon\alpha\sqrt{n\log n}$,
which for $\PP_{id}^{n}\left(\Pi_{t_{n}}\in\cdot\right)$ is strongly influenced
by $\left|\triangle_{\alpha}(\Pi_{t_n})\cap A^{t_{n}}\right|$, i.e. by the cards
yet-to-be-removed, is of different order for the two measures.

In Section \ref{sec:Position} we prove Theorem \ref{thm:Gaussian delocalization}
and other related results. We analyze the distribution of $\left|\triangle_{\alpha}(\sigma)\right|$
under $\UU^{n}$, and the distributions of $\left|\triangle_{\alpha}(\Pi_{t_n})\cap A^{t_{n}}\right|$
and $\left|\triangle_{\alpha}(\Pi_{t_n})\setminus A^{t_{n}}\right|$ under $\PP_{id}^{n}$ in Section \ref{sec:delta}.
The proof of Theorem \ref{thm:mixing}, given in Section \ref{sec:pfmix}, then easily follows.
Lastly, in Section \ref{sec:pflem} we prove a result which is used in the previous sections.

\section{\label{sec:Position}The Position of Cards Yet-to-be-Removed}

In this section we prove Theorem \ref{thm:Gaussian delocalization}
and other related results.

The increment distribution of $\Pi_{t}$ is given by
\begin{equation}
\mu\left(\tau\right)=\begin{cases}
1/n & \mbox{if }\tau=id,\\
2/n^{2} & \mbox{if }\tau=\left(i,j\right)\mbox{ with }1\leq i,j\leq n\mbox{ and }\left|i-j\right|=1,\\
1/n^{2} & \mbox{if }\tau=c_{i,j}\mbox{ with }1\leq i,j\leq n\mbox{ and }\left|i-j\right|>1,\\
0 & \mbox{otherwise,}
\end{cases}\label{eq:51}
\end{equation}
where $c_{i,j}$ is the cycle corresponding to removing the card in
position $i$ and reinserting it in position $j$, that is
\[
c_{i,j}=\begin{cases}
id & \mbox{if }i=j,\\
\left(j,j-1,\ldots,i+1,i\right) & \mbox{if }i<j,\\
\left(j,j+1,\ldots,i-1,i\right) & \mbox{if }i>j.
\end{cases}
\]

Let $2\leq n\in\mathbb{N}$ and $j\in\left[n\right]$. Under conditioning
on $\left\{ j\in A^{t}\right\} $, $\Pi_{m}\left(j\right)$, $m=0,\ldots,t$
is a time homogeneous Markov chain with transition probabilities
\begin{align}
p_{i,i+k}^{\Pi\left(j\right)} & \triangleq \PP_{{\scriptstyle id}}^{n}\left(\left.\Pi_{m+1}\left(j\right)=i+k\right|\Pi_{m}\left(j\right)=i,\, j\in A^{t}\right)\nonumber \\
 & =\begin{cases}
\frac{i\left(n-i\right)}{n\left(n-1\right)} & \mbox{if }k=+1,\\
\frac{\left(i-1\right)\left(n-i+1\right)}{n\left(n-1\right)} & \mbox{if }k=-1,\\
\frac{\left(i-1\right)^{2}+\left(n-i\right)^{2}}{n\left(n-1\right)} & \mbox{if }k=0,\\
0 & \mbox{ otherwise.}
\end{cases}\label{eq:48}
\end{align}

One of the difficulties in analyzing the chain is the fact that the
transition probabilities $p_{i,i+k}^{\Pi\left(j\right)}$ are inhomogeneous
in $i$. To overcome this, we consider a modification of the process
for which inhomogeneity is `truncated' by setting transition probabilities
far from the initial state to be identical to these in the initial
state. As we shall see, a bound on the TV distance of
the marginal distributions of the modified and original processes
is easily established.

For $j\in\left[n\right]$ and $M>0$, let $\overline{j\pm M}\triangleq\left[n\right]\cap\left[j-M,j+M\right]$,
and let $\zeta_{m}=\zeta_{m}^{n,j,M}$, $m=0,1,\ldots$ be a Markov
process starting at $\zeta_{0}=j$ with transition probabilities
$p_{i,i+k}^{\zeta,j,M}\triangleq \PP\left(\left.\zeta_{m+1}=i+k\right|\zeta_{m}=i\right)$
such that
\begin{align*}
\forall i\in\overline{j\pm M}:\,\, & p_{i,i+k}^{\zeta,j,M}=p_{i,i+k}^{\Pi\left(j\right)},\\
\forall i\in\mathbb{Z}\setminus\overline{j\pm M}:\,\, & p_{i,i+k}^{\zeta,j,M}=p_{j,j+k}^{\Pi\left(j\right)}.
\end{align*}

Clearly, for any sequence $\left(k_{m}\right)_{m=0}^{t}\in\mathbb{Z}^{t+1}$
if $\max_{0\leq m\leq t}\left|k_{m}-j\right|\leq M$ then
\begin{equation}
\PP\left(\left(\zeta_{m}\right)_{m=0}^{t}=\left(k_{m}\right)_{m=0}^{t}\right)=\PP_{{\scriptstyle id}}^{n}\left(\left.\left(\Pi_{m}\left(j\right)\right)_{m=0}^{t}=\left(k_{m}\right)_{m=0}^{t}\right|j\in A^{t}\right).\label{eq:45}
\end{equation}
Therefore, by taking complements, for any $u\leq M$
\begin{equation}
\PP_{{\scriptstyle id}}^{n}\left(\left.\max_{0\leq m\leq t}\left|\Pi_{m}\left(j\right)-j\right|>u\right|j\in A^{t}\right)=\PP\left(\max_{0\leq m\leq t}\left|\zeta_{m}-j\right|>u\right).\label{eq:31}
\end{equation}

Moreover, \eqref{eq:45} implies that for any $B\subset\mathbb{Z}^{t+1}$
\begin{align*}
 & \PP_{{\scriptstyle id}}^{n}\left(\left.\left(\Pi_{m}\left(j\right)\right)_{m=0}^{t}\in B\right|j\in A^{t}\right)-\PP\left(\left(\zeta_{m}\right)_{m=0}^{t}\in B\right)\\
 & \,\,\,\,=\PP_{{\scriptstyle id}}^{n}\left(\left.\left(\Pi_{m}\left(j\right)\right)_{m=0}^{t}\in B,\max_{0\leq m\leq t}\left|\Pi_{m}\left(j\right)-j\right|>M\right|j\in A^{t}\right)\\
 & \,\,\,\,\,\,\,\,-\PP\left(\left(\zeta_{m}\right)_{m=0}^{t}\in B,\max_{0\leq m\leq t}\left|\zeta_{m}-j\right|>M\right).
\end{align*}
Since both terms in the last equality are bounded from above by the
equal expressions of \eqref{eq:31} (and from below by zero), it follows
that

\begin{equation}
\label{eq:38}
\begin{aligned}
& \left\Vert \PP_{{\scriptstyle id}}^{n}\left(\left.\left(\Pi_{m}\left(j\right)\right)_{m=0}^{t}\in\cdot\right|j\in A^{t}\right)-\PP\left(\left(\zeta_{m}\right)_{m=0}^{t}\in\cdot\right)\right\Vert _{TV}\\
& \qquad \leq \PP\left(\max_{0\leq m\leq t}\left|\zeta_{m}-j\right|>M\right).
\end{aligned}
\end{equation}

A simple computation shows that $\left|p_{i,i+1}^{\Pi\left(j\right)}-p_{i,i-1}^{\Pi\left(j\right)}\right|$
is bounded by $\frac{1}{n}$ for any $i$. On the other hand, $p_{i,i\pm1}^{\Pi\left(j\right)}$
is roughly equal to $i\left(n-i\right)/n^{2}$. Thus if $j$ is large
enough and $M$, and thus $\left|\overline{j\pm M}\right|$, is small
compared to $j$, we can think of $\zeta_{m}^{n,j,M}$ as a perturbation
of a random walk with a very small bias. In order to make this precise,
we decompose $\zeta_{m}^{n,j,M}$ as a sum of a random walk determined
by the increment distribution in state $j$ and two additional random
processes related to the `defects' in symmetry and homogeneity in
state.

Consider the vector-valued Markov process $$\left(S_{m},X_{m},Y_{m}\right)=\left(S_{m}^{n,j,M},X_{m}^{n,j,M},Y_{m}^{n,j,M}\right)$$
starting at $\left(S_{0},X_{0},Y_{0}\right)=\left(0,0,0\right)$ with
transition probabilities as follows. For each $k\in\mathbb{Z}$ define
\begin{align}
q_{k}=&\, \min\left\{ p_{k,k+1}^{\zeta,j,M},p_{k,k-1}^{\zeta,j,M}\right\} ,\label{eq:34}\\
r_{k}=&\, \max\left\{ p_{k,k+1}^{\zeta,j,M},p_{k,k-1}^{\zeta,j,M}\right\} \nonumber.
\end{align}
For a state $\left(i_{1},i_{2},i_{3}\right)$ set $i=i_{1}+i_{2}+i_{3}$
and define
\begin{align}
\nonumber
w_{i}=&\, \arg\max_{k=\pm1}\left(p_{j+i,j+i+k}^{\zeta,j,M}\right),\\
z_{i}=&\,\, \mbox{sgn}\left(q_{j}-q_{j+i}\right),\nonumber
\end{align}
where sgn is the sign function (the definition of sgn at zero will
not matter to us). Define the transition probabilities by
\begin{align*}
 & \PP\left(\left.\left(S_{m+1},X_{m+1},Y_{m+1}\right)=\left(i_{1}+k_{1},i_{2}+k_{2},i_{3}+k_{3}\right)\right|\left(S_{m},X_{m},Y_{m}\right)=\left(i_{1},i_{2},i_{3}\right)\right)\\
 & \quad=\begin{cases}
\min\left\{ q_{j+i},q_{j}\right\}  & \mbox{if }\left(k_{1},k_{2},k_{3}\right)=\left(+1,0,0\right),\\
\min\left\{ q_{j+i},q_{j}\right\}  & \mbox{if }\left(k_{1},k_{2},k_{3}\right)=\left(-1,0,0\right),\\
\left|q_{j}-q_{j+i}\right| & \mbox{if }\left(k_{1},k_{2},k_{3}\right)=\left(+\frac{1+z_{i}}{2},-1,0\right),\\
\left|q_{j}-q_{j+i}\right| & \mbox{if }\left(k_{1},k_{2},k_{3}\right)=\left(-\frac{1+z_{i}}{2},+1,0\right),\\
r_{j+i}-q_{j+i}, & \mbox{if }\left(k_{1},k_{2},k_{3}\right)=\left(0,0,w_{i}\right),\\
c_{i}, & \mbox{if }\left(k_{1},k_{2},k_{3}\right)=\left(0,0,0\right).
\end{cases}
\end{align*}
where $c_{i}$ is chosen such that the sum of probabilities is $1$.

It is easy to verify that $\left(S_{m}+X_{m}+Y_{m}\right)_{m=0}^{\infty}$
is a Markov process with transition probabilities identical to those
of $\left(\zeta_{m}-j\right)_{m=0}^{\infty}$. Therefore the two processes
have the same law. It is also easy to check that $S_{n}$ is a random
walk with increment distribution
\[
\mu\left(+1\right)=\mu\left(-1\right)=q_{j}\,\,,\,\,\mu\left(0\right)=1-2q_{j}.
\]

In order to study $X_{m}$ and $Y_{m}$ we need the following proposition.
\begin{prop}
\label{pro:modification}Let $\left\{ A_{m}\right\} _{m=0}^{\infty}$
and $\left\{ B_{m}\right\} _{m=0}^{\infty}$ be integer-valued random
processes starting at the same point $A_{0}=B_{0}$. Suppose that
there exist $p_{ik}^{A}\in\left[0,1\right]$ such that for any $m\geq0$
and $k,i,i_{0},\ldots,i_{m-1}\in\mathbb{Z}$ (such that the conditional
probabilities are defined)
\begin{align*}
p_{ik}^{A} & =\PP\left(\left.A_{m+1}=k\right|A_{m+1}\neq i,A_{m}=i\right)\\
 & =\PP\left(\left.A_{m+1}=k\right|A_{m+1}\neq i,A_{m}=i,A_{m-1}=i_{m-1},\ldots,A_{0}=i_{0}\right)
\end{align*}
and similarly for $B_{m}$ with $p_{ik}^{B}$. Assume that for any
$i,k\in\mathbb{Z}$, $p_{ik}^{A}=p_{ik}^{B}$. Finally, suppose that
for any $m\geq0$ and $k,i,i_{0},\ldots,i_{m-1},j_{0},\ldots,j_{m-1}\in\mathbb{Z}$,
(whenever defined)
\begin{align*}
 & \PP\left(\left.A_{m+1}\neq i\right|A_{m}=i,A_{m-1}=i_{m-1},\ldots,A_{0}=i_{0}\right)\\
 & \,\,\,\,\geq \PP\left(\left.B_{m+1}\neq i\right|B_{m}=i,B_{m-1}=j_{m-1},\ldots,B_{0}=j_{0}\right).
\end{align*}
Then for any $t\in\mathbb{N}$ and $\delta>0$
\[
\PP\left(\max_{0\leq m\leq t}\left|A_{m}\right|\geq\delta\right)\geq \PP\left(\max_{0\leq m\leq t}\left|B_{m}\right|\geq\delta\right).
\]
\end{prop}
\begin{proof}
The processes $\left\{ A_{m}\right\} $ and $\left\{ B_{m}\right\} $
can be coupled so that they jump from a given state to a new state
according to the same order of states, say according to the order
$\left\{ k_{m}\right\} _{m=0}^{\infty}$, and such that the amount
of time that $\left\{ B_{m}\right\} $ spends in any given state $k_{m}$
before jumping to state $k_{m+1}$ is at least as much as $\left\{ A_{m}\right\} $
spends there. The proposition follows easily from this.
\end{proof}
The only nonzero increments of $X_{m}$ are $\pm1$. Note that
\begin{align*}
 & \PP\left(X_{m+1}=i_{m}+1\Big|X_{m+1}\neq i_{m},\left\{ X_{p}\right\} _{p=0}^{m}=\left\{ i_{p}\right\} _{p=0}^{m}\right)\\
 & =\sum_{\left(k_{1},k_{2}\right)\in\mathbb{Z}^{2}}\PP\left(X_{m+1}=i_{m}+1\Big|X_{m+1}\neq i_{m},\left\{ X_{p}\right\} _{p=0}^{m}=\left\{ i_{p}\right\} _{p=0}^{m},\ldots \right. \\
 & S_{m}=k_{1}, Y_{m}=k_{2} \Big) \PP\left(S_{m}=k_{1},Y_{m}=k_{2}\Big|X_{m+1}\neq i_{m},\left\{ X_{p}\right\} _{p=0}^{m}=\left\{ i_{p}\right\} _{p=0}^{m}\right)  \\
 & =\frac{1}{2}.
\end{align*}
The last equality follows from Markov property of $\left(S_{m},X_{m},Y_{m}\right)$.
The same, of course, holds for the negative increment. In addition,
again by Markov property,
\begin{align*}
 & \PP\left(\left.X_{m+1}\neq i_{m}\right|\left\{ X_{p}\right\} _{p=0}^{m}=\left\{ i_{p}\right\} _{p=0}^{m}\right)\\
 & \,\,\,\,=\sum_{\left(k_{1},k_{2}\right)\in\mathbb{Z}^{2}}\PP\left(X_{m+1}\neq i_{m}\Big|\left\{ X_{p}\right\} _{p=0}^{m}=\left\{ i_{p}\right\} _{p=0}^{m},S_{m}=k_{1},Y_{m}=k_{2}\right)\times\\
 & \,\,\,\,\,\,\,\, \PP\left(S_{m}=k_{1},Y_{m}=k_{2}\Big|\left\{ X_{p}\right\} _{p=0}^{m}=\left\{ i_{p}\right\} _{p=0}^{m}\right)\\
 & \,\,\,\,\leq\max_{k_{1},k_{2}}\PP\left(X_{m+1}\neq i_{m}\Big|X_{m}=i_{m},S_{m}=k_{1},Y_{m}=k_{2}\right)\leq2\max_{i\in\overline{j\pm M}}\left|q_{i}-q_{j}\right|\\
 & \,\,\,\,\leq2M\max_{x\in\left[1,n\right]}\left(\max\left\{ \left|\frac{d}{dx}\frac{x\left(n-x\right)}{n\left(n-1\right)}\right|,\left|\frac{d}{dx}\frac{\left(x-1\right)\left(n-x+1\right)}{n\left(n-1\right)}\right|\right\} \right)\\
 & \,\,\,\,\leq\frac{2M}{n-1},
\end{align*}
where the maximum in the first inequality is over all $k_{1},k_{2}$
such that the conditional probability is defined.

Thus, according to Proposition \ref{pro:modification}, for $\delta>0$,
\begin{equation}
\PP\left(\max_{0\leq m\leq t}\left|X_{m}^{n,j,M}\right|\geq\delta\right)\leq \PP\left(\max_{0\leq m\leq t}\left|W_{m}^{n,M}\right|\geq\delta\right),\label{eq:43}
\end{equation}
where $W_{m}=W_{m}^{n,M}$ is a random walk starting at $0$ with
increment distribution
\[
\nu\left(+1\right)=\nu\left(-1\right)=\frac{M}{n-1}\,\,,\,\,\nu\left(0\right)=1-2\frac{M}{n-1}.
\]

Similarly, for the process $\widetilde{Y}_{t}=\sum_{m=1}^{t}\left|Y_{m}-Y_{m-1}\right|$,
whose increments are $0$ and $1$, we have
\begin{align*}
 & \PP\left(\left.\widetilde{Y}_{m+1}=i_{m}+1 \, \right|\left\{ \widetilde{Y}_{p}\right\} _{p=0}^{m}=\left\{ i_{p}\right\} _{p=0}^{m}\right)\\
 & \,\,\,\,\leq\max_{k_{1},k_{2},k_3}\PP\left(\left.Y_{m+1}\neq k_1\right|Y_{m}=k_1,S_{m}=k_2,X_{m}=k_3\right)\\
 & \,\,\,\,\leq\max_{i\in\mathbb{Z}}\left(r_{j+i}-q_{j+i}\right)=\max_{i\in\overline{j\pm M}}\left|\frac{i\left(n-i\right)}{n\left(n-1\right)}-\frac{\left(i-1\right)\left(n-i+1\right)}{n\left(n-1\right)}\right|\\
 & \,\,\,\,=\max_{i\in\overline{j\pm M}}\left|\frac{n-2i+1}{n\left(n-1\right)}\right|\leq\frac{1}{n}.
\end{align*}
Therefore, for $\delta>0$,
\begin{equation}
\PP\left(\max_{0\leq m\leq t}\left|Y_{m}^{n,j,M}\right|\geq\delta\right)\leq \PP\left(\widetilde{Y}_{t}\geq\delta\right)\leq \PP\left(N_{t}^{n}\geq\delta\right),\label{eq:44}
\end{equation}
where $N_{t}=N_{t}^{n}\sim\mbox{Bin}\left(t,\frac{1}{n}\right)$.

Since the increment distributions of $W_{m}$ and $S_{m}$ are symmetric,
the classical Lévy inequality (\cite{pPET95a}, Theorem 2.2)  yields, for any $\delta>0$,
\begin{equation}
\PP\left(\max_{0\leq m\leq t}\left|W_{m}\right|\geq\delta\right)\leq4\PP\left(W_{t}\geq\delta\right).\label{eq:40}
\end{equation}
and
\begin{equation}
\PP\left(\max_{0\leq m\leq t}\left|S_{m}\right|\geq\delta\right)\leq4\PP\left(S_{t}\geq\delta\right).\label{eq:46}
\end{equation}

Having established the connections between the different processes,
we are now ready to prove Theorem \ref{thm:Gaussian delocalization}.

\begin{proof}
(Theorem \ref{thm:Gaussian delocalization}) The case where $\gamma=1$ follows by symmetry from the case with
$\gamma=0$. Assume $\gamma\in\left[0,1\right)$. In this case, the
hypothesis in the theorem are equivalent to $$\lim_{n\rightarrow\infty}\frac{n}{t_{n}j_{n}}=\lim_{n\rightarrow\infty}\frac{t_{n}}{nj_{n}}=0.$$

Let $n\in\mathbb{N}$, $j\in\left[n\right]$ and $M>0$. Based on \eqref{eq:43}-\eqref{eq:40} and a union bound, for $u\in\mathbb{R}$,
$\delta>0$, we have
\begin{align*}
& \PP\left(\zeta_{t}-j\geq u\right) \\
& \quad \leq \PP\left(S_{t}\geq u-\delta\right)+\PP\left(\max_{0\leq m\leq t}\left|X_{m}\right|\geq\frac{\delta}{2}\right)+\PP\left(\max_{0\leq m\leq t}\left|Y_{m}\right|\geq\frac{\delta}{2}\right)\\
 & \quad \leq \PP\left(S_{t}\geq u-\delta\right)+4\PP\left(W_{t}\geq\frac{\delta}{2}\right)+\PP\left(N_{t}\geq\frac{\delta}{2}\right).
\end{align*}
Similarly,
\begin{align*}
& \PP\left(\zeta_{t}-j\geq u\right) \\
 & \quad \geq \PP\left(S_{t}\geq u+\delta\right)-\PP\left(\max_{0\leq m\leq t}\left|X_{m}\right|\geq\frac{\delta}{2}\right)-\PP\left(\max_{0\leq m\leq t}\left|Y_{m}\right|\geq\frac{\delta}{2}\right)\\
 & \quad \geq \PP\left(S_{t}\geq u+\delta\right)-4\PP\left(W_{t}\geq\frac{\delta}{2}\right)-\PP\left(N_{t}\geq\frac{\delta}{2}\right).
\end{align*}

Assume $\frac{\delta}{2}-\frac{t}{n}>0$. By computing moments and
applying the Berry-Esseen theorem to approximate the tail probability
function of $S_{t}$, and applying Chebyshev's inequality to bound
the tail probability functions of $W_{t}$ and $N_{t}$, we arrive
at
\begin{equation}
\PP\left(\zeta_{t}^{n,j,M}-j\geq u\right)\leq\Psi\left(\frac{u-\delta}{\sqrt{2tq_{j}}}\right)+\frac{C}{\sqrt{2tq_{j}}}+\frac{32Mt}{\delta^{2}\left(n-1\right)}+\frac{\frac{t}{n}\frac{n-1}{n}}{\left(\frac{\delta}{2}-\frac{t}{n}\right)^{2}}\label{eq:35}
\end{equation}
and
\begin{equation}
\PP\left(\zeta_{t}^{n,j,M}-j\geq u\right)\geq\Psi\left(\frac{u+\delta}{\sqrt{2tq_{j}}}\right)-\frac{C}{\sqrt{2tq_{j}}}-\frac{32Mt}{\delta^{2}\left(n-1\right)}-\frac{\frac{t}{n}\frac{n-1}{n}}{\left(\frac{\delta}{2}-\frac{t}{n}\right)^{2}},\label{eq:36}
\end{equation}
where $q_{j}$ is defined in \eqref{eq:34}, $C$ is the constant
from the Berry-Esseen theorem and $\Psi$ is the tail probability
function of a standard normal variable.

For two sequences of positive numbers $v_{n},v_{n}^{\prime}$ let
us denote $v_{n}\ll v_{n}^{\prime}$ if and only if $\lim_{n\rightarrow\infty}v_{n}/v_{n}^{\prime}=0$.
By assumption, $\sqrt{\frac{t_{n}j_{n}}{n}}\ll j_{n}$, therefore
we can choose a sequence $M_{n}$ such that $\frac{t_{n}}{n}, 1\ll \sqrt{\frac{t_{n}j_{n}}{n}}\ll M_{n}\ll j_{n}$.
Similarly, since $M_{n}\ll j_{n}$ we can set $\delta_{n}$ with $\sqrt{\frac{t_{n}M_{n}}{n}}\ll\delta_{n}\ll\sqrt{\frac{t_{n}j_{n}}{n}}$,
which also implies that $\sqrt{\frac{t_{n}}{n}},\frac{t_{n}}{n}\ll\delta_{n}$.

Now, let $x\in\mathbb{R}$ and set $u_{n}=x\sqrt{2t_{n}\lambda_{n}}$.
Let us consider the inequalities derived from \eqref{eq:35} and \eqref{eq:36}
by replacing each of the parameters by a corresponding element from
the sequences above. Based on the relations established for the sequences
and the assumptions on $t_{n}$ and $j_{n}$ it can be easily verified
that, upon letting $n\rightarrow\infty$, all terms but those involving
$\Psi$ go to zero. Relying, in addition, on the fact that $\Psi$
is continuous, it can be easily verified that
\[
\lim_{n\rightarrow\infty}\Psi\left(\frac{u_{n}\pm\delta_{n}}{\sqrt{2t_{n}q_{j_{n}}}}\right)=\Psi\left(x\right).
\]
Hence we conclude that
\begin{equation}
\lim_{n\rightarrow\infty}\PP\left(\zeta_{t_{n}}^{n,j_{n},M_{n}}-j_{n}\geq u_{n}\right)=\Psi\left(x\right).\label{eq:39}
\end{equation}

Based on \eqref{eq:43}-\eqref{eq:46},
\begin{equation}
\label{eq:zetamax}
\begin{aligned}
 & \PP\left(\max_{0\leq m\leq t}\left|\zeta_{m}^{n,j,M}-j\right|\geq M\right) \leq \PP\left(\max_{0\leq m\leq t}\left|S_{m}\right|\geq M-\delta\right)\\
  & \quad \quad  +\PP\left(\max_{0\leq m\leq t}\left|Y_{m}\right|\geq\frac{\delta}{2}\right)+\PP\left(\max_{0\leq m\leq t}\left|X_{m}\right|\geq\frac{\delta}{2}\right)\\
  & \quad \leq4\PP\left(S_{t}\geq M-\delta\right)+4\PP\left(W_{t}\geq\frac{\delta}{2}\right)+\PP\left(N_{t}\geq\frac{\delta}{2}\right) \\
  &\quad  \leq\frac{8tq_{j}}{\left(M-\delta\right)^{2}}+\frac{32Mt}{\delta^{2}\left(n-1\right)}+\frac{\frac{t}{n}\frac{n-1}{n}}{\left(\frac{\delta}{2}-\frac{t}{n}\right)^{2}},
\end{aligned}
\end{equation}
where the last inequality follows from Chebyshev's inequality.

As before, consider the inequality derived from \eqref{eq:zetamax} by replacing each of the parameters by a corresponding
element from the sequences above. The middle and right-hand
side summands of \eqref{eq:zetamax} were already shown to go to zero as
$n\rightarrow\infty$. Since $\delta_{n}\ll\sqrt{\frac{t_{n}j_{n}}{n}}\ll M_{n}$
the additional term also goes to zero. Combined with \eqref{eq:38}
and \eqref{eq:39} this gives
\[
\lim_{n\rightarrow\infty}\PP_{{\scriptstyle id}}^{n}\left(\left.\Pi_{t_{n}}\left(j_{n}\right)-j_{n}\geq u_{n}\right|j_{n}\in A^{t_{n}}\right)=\Psi\left(x\right),
\]
which completes the proof.
\end{proof}

In Theorem \ref{thm:Gaussian delocalization} for each $n$ only a
single card $j_{n}$ of the deck of size $n$ is involved. The following
gives a uniform bound (in initial position and in time) for the tail
distributions of the difference from the initial position.
\begin{thm}
\label{thm:delocaliztion}Let $\alpha>0$ and let $t_{n}$ be a sequence of natural numbers
such that $\lim_{n\rightarrow\infty}t_{n}=\lim_{n\rightarrow\infty}n^{2}/t_{n}=\infty$.
Then
\[
\limsup_{n\rightarrow\infty}\max_{j\in\left[n\right]}\PP_{id}^{n}\left(\left.\max_{0\leq m\leq t_{n}}\left|\Pi_{m}\left(j\right)-j\right|>\alpha\sqrt{\frac{t_{n}}{2}} \, \right|j\in A^{t_{n}}\right)\leq4\Psi\left(\alpha\right).
\]
\end{thm}
\begin{proof}
Set $u_{n}=\alpha\sqrt{\frac{t_{n}}{2}}$ and $j_{n}=\left\lfloor \frac{n}{2}\right\rfloor $.
Let $M_{n}\geq u_{n}$ and $\delta_{n}>0$ be sequences to be determined
below. From \eqref{eq:31} it follows that
\begin{align*}
 & \max_{j\in\left[n\right]}\PP_{id}^{n}\left(\left.\max_{0\leq m\leq t_{n}}\left|\Pi_{m}\left(j\right)-j\right|>u_{n}\right|j\in A^{t_{n}}\right)\\
 & \,\,\,\,=\max_{j\in\left[n\right]}\PP\left(\max_{0\leq m\leq t_{n}}\left|\zeta_{m}^{n,j,M_{n}}-j\right|>u_{n}\right)\\
 & \,\,\,\,\leq\max_{j\in\left[n\right]}\left\{ \PP\left(\max_{0\leq m\leq t_{n}}\left|S_{m}^{n,j,M_{n}}\right|\geq u_{n}-\delta_{n}\right)+\right.\\
 & \,\,\,\,\,\,\,\,\left.\PP\left(\max_{0\leq m\leq t_{n}}\left|W_{m}^{n,M_{n}}\right|\geq\frac{\delta_{n}}{2}\right)+\PP\left(N_{t_{n}}^{n}\geq\frac{\delta_{n}}{2}\right)\right\} .
\end{align*}

It is easy to check that for the random walks $S_{m}^{n,j,M_{n}}$,
$j\in\left[n\right]$, the probabilities of the nonzero increments,
$\pm1$, are maximal when $j=j_{n}$. Therefore according to Proposition
\ref{pro:modification} and equations \eqref{eq:40} and \eqref{eq:46},
\begin{align*}
 & \max_{j\in\left[n\right]}\PP_{id}^{n}\left(\left.\max_{0\leq m\leq t_{n}}\left|\Pi_{m}\left(j\right)-j\right|>u_{n}\right|j\in A^{t_{n}}\right)\\
 & \,\,\,\,\leq \PP\left(\max_{0\leq m\leq t_{n}}\left|S_{m}^{n,j_{n},M_{n}}\right|\geq u_{n}-\delta_{n}\right)+\\
 & \qquad \PP\left(\max_{0\leq m\leq t_{n}}\left|W_{m}^{n,M_{n}}\right|\geq\frac{\delta_{n}}{2}\right)+\PP\left(N_{t_{n}}^{n}\geq\frac{\delta_{n}}{2}\right) \\
 & \,\,\,\,\leq4\left\{ \PP\left(S_{t_{n}}^{n,j_{n},M_{n}}\geq u_{n}-\delta_{n}\right)+\PP\left(W_{t_{n}}^{n,M_{n}}\geq\frac{\delta_{n}}{2}\right)+\PP\left(N_{t_{n}}^{n}\geq\frac{\delta_{n}}{2}\right)\right\} .
\end{align*}

Finally, note that $j_{n}$ and $t_{n}$ meet the conditions of Theorem
\ref{thm:Gaussian delocalization} with $\gamma=\frac{1}{2}$. Therefore,
defining $M_{n}$ and $\delta_{n}$ as in the proof of the theorem
(which also implies $M_{n}\geq u_{n}$) and following the same arguments
therein, as $n \rightarrow \infty$ the expression in the last line of the inequality above converges to
\begin{equation}
\nonumber
4\Psi\left(\lim_{n\rightarrow\infty}\frac{u_{n}-\delta_n}{\sqrt{t_{n}/2}}\right)=4\Psi\left(\alpha\right).
\end{equation}
This completes the proof.
\end{proof}

%
%
%

\section{\label{sec:delta}Cards of Distance  $O(\sqrt{n\log n})$ from their Initial Position}

The results of Section \ref{sec:Position} show that the position
of a card that has not been removed is fairly concentrated around
the initial position. This, of course, is a rare event for each card
under the uniform measure $\UU^{n}$. In this section we shall develop the tools
to exploit this to derive a lower bound for the TV distance between
$\UU^{n}$ and $\PP_{id}^{n}\left(\Pi_{t}\in\cdot\right)$ whenever sufficiently
many (in expectation) of the cards have not been removed. Here, `sufficiently many' means, of course,
that $t$ is not too large.

More precisely, we shall consider the size of sets of the form
\[
\triangle_{\alpha}\left(\sigma\right)\triangleq\left\{ j\in D^{n}:\,\left|\sigma\left(j\right)-j\right|\leq\alpha\sqrt{n\log n}\right\} \,\,,\,\,\sigma\in S_{n},
\]
where $D^{n}=\left[n\right]\cap\left[n\left(1-\varepsilon\right)/2,n\left(1+\varepsilon\right)/2\right]$,
and $\varepsilon\in\left(0,1\right)$ is arbitrary and will be fixed
throughout the proofs. Under $\UU^{n}$, for $i\neq j$, the events
$\left\{ i\in\triangle_{\alpha}\right\} $ and $\left\{ j\in\triangle_{\alpha}\right\} $
are `almost' independent, as $n\rightarrow\infty$. Therefore one
should expect $\left|\triangle_{\alpha}\right|-\EE^{\UU^n} \left\{ \left|\triangle_{\alpha}\right|\right\} $
to be of order $\left(\EE^{\UU^{n}}\left\{ \left|\triangle_{\alpha}\right|\right\} \right)^{1/2}$.
Under $\PP_{id}^{n}$, if $\left|A^{t_{n}}\right|$
is relatively small, it seems natural that the positions of the cards
that have been removed are distributed approximately as they would
under $\UU^{n}$. Thus, $\left|\triangle_{\alpha}(\Pi_{t_{n}}) \setminus A^{t_{n}}\right|$
under $\PP_{id}^{n}$ should be distributed roughly as $\left|\triangle_{\alpha}\right|$ is under $\UU^{n}$. By
this logic, we need to choose $t_{n}$ so that
 $\left|\triangle_{\alpha}(\Pi_{t_{n}})\cap A^{t_{n}}\right|$ is larger
than $\left(\EE^{\UU^{n}}\left\{ \left|\triangle_{\alpha}\right|\right\} \right)^{1/2}$
with high probability, which leads us to set $t_{n}$ to be as in Theorem \ref{thm:mixing}.

The three subsections below are devoted to separately study the distribution of $\left|\triangle_{\alpha}\right|$
under $\UU^{n}$ and the distributions of $\left|\triangle_{\alpha}(\Pi_{t_{n}})\cap A^{t_{n}}\right|$
and $\left|\triangle_{\alpha}(\Pi_{t_{n}})\setminus A^{t_{n}}\right|$ under $\PP_{id}^{n}$.

\subsection{\label{subsec:deltaU} The distribution of $\left|\triangle_{\alpha}\right|$
under $\UU^{n}$}

In this case, the first and second moments of $\left|\triangle_{\alpha}\right|$ can be easily computed in order to apply Chebyshev's inequality.
In what follows, let $R_{j}$ denote the event $\left\{ j\in\triangle_{\alpha}\right\} $.
\begin{lem}
\label{lem:delta-U}For any $\alpha$, $k>0$,
\[
\limsup_{n\rightarrow\infty}\UU^{n}\left(\left|\left|\triangle_{\alpha}\left(\sigma\right)\right|-2\varepsilon\alpha\sqrt{n\log n}\right|\geq k\sqrt{2\varepsilon\alpha}\left(n\log n\right)^{\frac{1}{4}}\right)\leq\frac{1}{k^{2}}.
\]
\end{lem}
\begin{proof}
Suppose $n$ is large enough so that $n\left(1-\varepsilon\right)/2\geq\alpha\sqrt{n\log n}$.
Then
\begin{align*}
\EE^{\UU^{n}}\left\{ \left|\triangle_{\alpha}\left(\sigma\right)\right|\right\}  & =\sum_{j\in D^{n}}\UU^{n}\left(R_{j}\right)=\left|D^{n}\right|\frac{1+2\left\lfloor \alpha\sqrt{n\log n}\right\rfloor }{n}\\
 & =2\varepsilon\alpha\sqrt{n\log n}+O\left(1\right).
\end{align*}

The second moment satisfies the bound
\begin{align*}
\EE^{\UU^{n}}\left\{ \left|\triangle_{\alpha}\left(\sigma\right)\right|^{2}\right\}  & =\sum_{j\in D^{n}}\UU^{n}\left(R_{j}\right)+\sum_{i,j\in D^{n}:i\neq j}\UU^{n}\left(R_{i}\cap R_{j}\right)\\
 & \leq \EE^{\UU^{n}}\left\{ \left|\triangle_{\alpha}\left(\sigma\right)\right|\right\} +\left|D^{n}\right|^{2}\frac{\left(1+2\left\lfloor \alpha\sqrt{n\log n}\right\rfloor \right)^{2}}{n\left(n-1\right)}\\
 & =\EE^{\UU^{n}}\left\{ \left|\triangle_{\alpha}\left(\sigma\right)\right|\right\} +\frac{n}{n-1}\left(\EE^{\UU^{n}}\left\{ \left|\triangle_{\alpha}\left(\sigma\right)\right|\right\} \right)^{2},
\end{align*}
which implies
\begin{equation}
 \VVar^{\UU^{n}}\left\{ \left|\triangle_{\alpha}\left(\sigma\right)\right|\right\}\leq2\varepsilon\alpha\sqrt{n\log n}+4\varepsilon^{2}\alpha^{2}\log n+O\left(1\right).\nonumber
\end{equation}

Applying Chebyshev's inequality and letting $n\rightarrow\infty$
yields the required result.
\end{proof}

\subsection{\label{subsec:deltaP1} The distribution of $\left|\triangle_{\alpha}(\Pi_{t_{n}})\cap A^{t_{n}}\right|$
 under $\PP_{id}^{n}$}

We begin with the following lemma which, when combined with Theorem \ref{thm:delocaliztion}, yields a bound on the probability
that $\left|\triangle_{\alpha}(\Pi_{t_n})\cap A^{t_{n}}\right|$ is less than a fraction of its expectation.
This bound is the content of Lemma \ref{lem:delta-A} which takes up the rest of the subsection.

\begin{lem}
\label{lem:moment}Let $n,t\in\mathbb{N}$, let $B\subset\left[n\right]$
be a random set, and let $D\subset\left[n\right]$ be a deterministic
set. Suppose that for some $c>0$
\[
\min_{j\in D}\PP_{id}^{n}\left(\left.j\in B\right|j\in A^{t}\right)\geq c.
\]
Then, denoting $K=\EE_{id}^{n}\left|D\cap A^{t}\right|$, for any $r\in\left(0,1\right)$,
\begin{align*}
 \PP_{id}^{n}\left(\left|B\cap D\cap A^{t}\right|\leq r\cdot \EE_{id}^{n}\left\{ \left|B\cap D\cap A^{t}\right|\right\} \right) \leq\frac{K+\left(1-c^{2}\right)K^{2}}{\left(1-r\right)^{2}c^{2}K^{2}}.
\end{align*}
\end{lem}

\begin{proof}
By our assumption,
\[
\EE_{id}^{n}\left\{ \left|B\cap D\cap A^{t}\right|\right\} =\sum_{j\in D}\PP_{id}^{n}\left(\left.j\in B\right|j\in A^{t}\right)\PP_{id}^{n}\left(j\in A^{t}\right)\geq cK.
\]

Write
\begin{align*}
 & \EE_{id}^{n}\left\{ \left|B\cap D\cap A^{t}\right|^{2}\right\} \leq \EE_{id}^{n}\left\{ \left|D\cap A^{t}\right|^{2}\right\}\\
 & \,\,\,\,=\sum_{j\in D}\PP_{id}^{n}\left(j\in A^{t}\right)+\sum_{i,j\in D:i\neq j}\PP_{id}^{n}\left(i,j\in A^{t}\right)\\
 & \,\,\,\,=K+\left|D\right|\left(\left|D\right|-1\right)\left(\frac{n-2}{n}\right)^{t}.
\end{align*}
Since $K= |D|((n-1)/n)^t$ it follows that
\[
\EE_{id}^{n}\left\{ \left|B\cap D\cap A^{t}\right|^{2}\right\} \leq K+K^{2},
\]
therefore
\[
\VVar_{id}^{n}\left\{ \left|B\cap D\cap A^{t}\right|\right\} \leq K+\left(1-c^{2}\right)K^{2}.
\]

Applying Chebyshev's inequality completes the proof.

\end{proof}

Now, let $R_{j,t}$ and $R_{j,t}^{A^{c}}$ denote the events $\left\{ j\in\triangle_{\alpha}\left(\Pi_{t}\right)\right\} $
and $\left\{ j\in\triangle_{\alpha}\left(\Pi_{t}\right)\right\} \cap\left\{ j\notin A^{t}\right\} $,
respectively. Let $p_{t,n}\triangleq \PP_{id}^{n}\left(j\in A^{t}\right)$
(which is, of course, independent of $j$).
\begin{lem}
\label{lem:delta-A}Let $v\left(\alpha\right)=1-4\Psi\left(\alpha\sqrt{\frac{8}{3}}\right)$.
Let $t_n$ be a sequence of natural numbers such that $t_{n}\leq\frac{3}{4}n\log n$ and suppose $\alpha$ satisfies
$v\left(\alpha\right)>0$. Then, for any $r\in\left(0,1\right)$,
\begin{align*}
 & \limsup_{n\rightarrow\infty}\PP_{id}^{n}\left(\left|\triangle_{\alpha}\left(\Pi_{t_{n}}\right)\cap A^{t_{n}}\right|\leq rv\left(\alpha\right)\varepsilon np_{t_{n},n}\right)\\
 & \qquad\leq\left(1-r\right)^{-2}\left(v^{-2}\left(\alpha\right)-1\right).
\end{align*}
\end{lem}
\begin{proof}
With $\mathfrak{\mathcal{S}}\left(n,\alpha\right)$ defined by
\begin{align*}
 & \mathfrak{\mathcal{S}}\left(n,\alpha\right)\\
 & \,\,\triangleq\min_{j\in D^{n}}\PP_{id}^{n}\left(\left.\max_{0\leq m\leq\frac{3}{4}n\log n}\left|\Pi_{m}\left(j\right)-j\right|\leq\alpha\sqrt{n\log n}\right|j\in A^{\left\lfloor \frac{3}{4}n\log n\right\rfloor }\right)\\
 & \,\,\leq\min_{j\in D^{n}}\PP_{id}^{n}\left(\left.R_{j,t_{n}}\right|j\in A^{\left\lfloor \frac{3}{4}n\log n\right\rfloor }\right)\\
 & \,\,=\min_{j\in D^{n}}\PP_{id}^{n}\left(\left.R_{j,t_{n}}\right|j\in A^{t_{n}}\right),
\end{align*}
Lemma \ref{lem:moment} yields
\begin{align*}
 & \PP_{id}^{n}\left(\left|\triangle_{\alpha}\left(\Pi_{t_{n}}\right)\cap A^{t_{n}}\right|\leq r\cdot \EE_{id}^{n}\left\{ \left|\triangle_{\alpha}\left(\Pi_{t_{n}}\right)\cap A^{t_{n}}\right|\right\} \right)\\
 & \qquad\leq\frac{K_{t_{n}}+\left(1-\mathfrak{\mathcal{S}}^{2}\left(n,\alpha\right)\right)K_{t_{n}}^{2}}{\left(1-r\right)^{2}\mathfrak{\mathcal{S}}^{2}\left(n,\alpha\right)K_{t_{n}}^{2}},
\end{align*}
where $K_{t_{n}}\triangleq \EE_{id}^{n}\left\{ \left|D^{n}\cap A^{t_{n}}\right|\right\} .$

A simple calculation shows that $\lim_{n\rightarrow\infty}K_{t_{n}}=\infty$.
Theorem \ref{thm:delocaliztion} (with $t_{n}=\left\lfloor \frac{3}{4}n\log n\right\rfloor $)
implies that
\begin{equation}
\liminf_{n\rightarrow\infty}\mathfrak{\mathcal{S}}\left(n,\alpha\right)\geq1-4\Psi\left(\alpha\sqrt{\frac{8}{3}}\right)=v\left(\alpha\right)>0.\label{eq:5-1}
\end{equation}
Therefore
\begin{equation}
\label{eq:5}
\begin{aligned}
 & \limsup_{n\rightarrow\infty}\PP_{id}^{n}\left(\left|\triangle_{\alpha}\left(\Pi_{t_{n}}\right)\cap A^{t_{n}}\right|\leq r\cdot \EE_{id}^{n}\left\{ \left|\triangle_{\alpha}\left(\Pi_{t_{n}}\right)\cap A^{t_{n}}\right|\right\} \right)\\
 & \qquad\leq\frac{1}{\left(1-r\right)^{2}}\limsup_{n\rightarrow\infty}\frac{\left(1-\mathfrak{\mathcal{S}}^{2}\left(n,\alpha\right)\right)}{\mathfrak{\mathcal{S}}^{2}\left(n,\alpha\right)} \\
 & \qquad\leq\left(1-r\right)^{-2}\left(v^{-2}\left(\alpha\right)-1\right).
\end{aligned}
\end{equation}

Note that by \eqref{eq:5-1}, for any $\delta\in\left(0,1\right)$
and sufficiently large $n$,
\begin{align*}
\EE_{id}^{n}\left\{ \left|\triangle_{\alpha}\left(\Pi_{t_{n}}\right)\cap A^{t_{n}}\right|\right\}  & =\sum_{j\in D^{n}}\PP_{id}^{n}\left(\left.R_{j,t_{n}}\right|j\in A^{t_{n}}\right)\PP_{id}^{n}\left(j\in A^{t_{n}}\right)\\
 & \geq\left|D^{n}\right|p_{t,n}\mathfrak{\mathcal{S}}\left(n,\alpha\right)\geq\delta \varepsilon np_{t_{n},n} v\left(\alpha\right).
\end{align*}

Together with \eqref{eq:5}, this implies
\begin{align*}
 & \limsup_{n\rightarrow\infty}\PP_{id}^{n}\left(\left|\triangle_{\alpha}\left(\Pi_{t_{n}}\right)\cap A^{t_{n}}\right|\leq rv\left(\alpha\right)\varepsilon np_{t_{n},n}\right)\\
 & \qquad\leq\left(1-r/\delta\right)^{-2}\left(v^{-2}\left(\alpha\right)-1\right).
\end{align*}

By letting $\delta\rightarrow1$, the lemma follows.
\end{proof}

\subsection{\label{subsec:deltaP2} The distribution of $\left|\triangle_{\alpha}(\Pi_{t_{n}}) \setminus A^{t_{n}}\right|$
 under $\PP_{id}^{n}$}

As in Subsection \ref{subsec:deltaU}, we shall use Chebyshev's inequality to bound the deviation
of $\left|\triangle_{\alpha}(\Pi_{t_n}) \setminus A^{t_{n}}\right|$ from its expectation with high probability. Here, however,
the computations are much more involved, and the main difficulty is to compute probabilities that depend to joint
distributions of $\Pi_t(i)$ and $\Pi_t(j)$, for general $i\neq j$ in $D^n$, conditioned on $i$ and $j$ not being in $A^t$.
This is treated in the lemma below, in which we denote by $\tau^t_m$ the last time up to time $t$ at which the card numbered $m$ is
chosen for removal, and set $\tau^t_m=\infty$ if it is not chosen up to that time.

\begin{lem}
\label{lem:2ndordProb}Let $i,j\in\left[n\right]$ such that $i\neq j$,
let $\delta >0$, and let $1\leq t_{1}<t_{2}\leq t$ be natural numbers with $t\leq n\log n$.
Then
\begin{align*}
\PP_{id}^{n}\left(\left.\left(\Pi_{t}\left(i\right),\Pi_{t}\left(j\right)\right)\in\overline{i\pm\delta}\times\overline{j\pm\delta}\,\right|\,\tau_{i}^{t}=t_{1},\tau_{j}^{t}=t_{2}\right)\\
\,\,\,\,\leq\frac{1}{n^{2}}\left(\left(1+2\left\lfloor \delta\right\rfloor \right)^{2}+4\delta+g\left(n\right)\right),
\end{align*}
where $g\left(n\right)=\Theta\left(\log^{2}n\right)$ is a function
independent of all the parameters above.\end{lem}

The proof of Lemma \ref{lem:2ndordProb} is given in Section \ref{sec:pflem}. Now,
let us see how it is used to prove the following.

\begin{lem}
\label{lem:delta-A^c}Let $t_n$ be a sequence of integers such that $0 \leq t_{n}\leq\left\lfloor \frac{3}{4}n\log n\right\rfloor$
 and let $k,\,\alpha>0$ be real numbers. Then,
\begin{align*}
 & \limsup_{n\rightarrow\infty}\PP_{id}^{n}\left(\left|\left|\triangle_{\alpha}\left(\Pi_{t_{n}}\right)\setminus A^{t_{n}}\right|-2\varepsilon\alpha\left(1-p_{t_{n},n}\right)\sqrt{n\log n}\right|\ldots\right.\\
 & \quad\quad\left.\ldots\geq k \sqrt{6\varepsilon\alpha} (n\log n)^{1/4}   \right)\leq\frac{1}{k^{2}}.
\end{align*}
\end{lem}

\begin{proof}
For some $t_1 \leq t$, consider the Markov chain  $\Pi_{t'}\left(j\right)$, $t'=t_1,\ldots,t$, conditioned on $\tau^t_j = t_1$. By definition,
its initial distribution is the uniform measure on $[n]$. One can easily check that the transition matrix of this chain is symmetric. Therefore its
stationary measure, and thus its distribution at time $t$, is also the uniform measure.

Thus, assuming $n(1-\varepsilon)/2\geq\alpha\sqrt{n\log n}$,
\begin{align*}
\EE_{id}^{n}\left\{ \left|\triangle_{\alpha}\left(\Pi_{t}\right)\setminus A^{t}\right|\right\}  & =\sum_{j\in D^{n}}\PP_{id}^{n}\left(\left.R_{j,t}^{A^{c}}\right|j\notin A^{t}\right)\PP_{id}^{n}\left(j\notin A^{t}\right)\\
 & =\left|D^{n}\right|\frac{1+2\left\lfloor \alpha\sqrt{n\log n}\right\rfloor }{n}\left(1-p_{t,n}\right).
\end{align*}

For the second moment write
\begin{equation}
\nonumber
\EE_{id}^{n}\left\{ \left|\triangle_{\alpha}\left(\Pi_{t}\right)\setminus A^{t}\right|^{2}\right\}  =\EE_{id}^{n}\left\{ \left|\triangle_{\alpha}\left(\Pi_{t}\right)\setminus A^{t}\right|\right\} +\sum_{i,j\in D:i\neq j}\PP_{id}^{n}\left(R_{i,t}^{A^{c}}\cap R_{j,t}^{A^{c}}\right).
\end{equation}

From Lemma \ref{lem:2ndordProb},
\begin{align*}
& \sum_{i,j\in D:i\neq j}\PP_{id}^{n}\left(R_{i,t}^{A^{c}}\cap R_{j,t}^{A^{c}}\right) \leq \frac{\left|D^{n}\right|^2}{n^{2}} \left(1-p_{t,n}\right)^2 \times \\
&\qquad \left\{\left(1+2\left\lfloor \alpha\sqrt{n\log n}\right\rfloor \right)^{2}+4\alpha\sqrt{n\log n}+\Theta\left(\log^{2}n\right)\right\}.
\end{align*}

Therefore,
\begin{equation}
\nonumber
\limsup_{n\rightarrow\infty} \frac{\VVar_{id}^{n}\left\{ \left|\triangle_{\alpha}\left(\Pi_{t_{n}}\right)\setminus A^{t_n}\right|\right\}}{\sqrt{n\log n}}
\leq 4\varepsilon^2 \alpha + 2\varepsilon \alpha \leq 6\varepsilon \alpha.
\end{equation}

By Chebyshev's inequality, the lemma follows.
\end{proof}

\begin{remark}
Assume $t_n$ is of the form in Theorem \ref{thm:mixing} with $c_n$ satisfying $\limsup c_n/ \log n < 1/4$.
From Lemma \ref{lem:delta-A^c},
$$\left|\triangle_{\alpha}(\Pi_{t_n})\setminus A^{t_n}\right|/\left(2\varepsilon\alpha\sqrt{n\log n}\right) \Longrightarrow 1.$$

By a simple computation, taking into account our restriction on $c_n$, it is seen that $\EE_{id}^{n}\left|A^{t_n}\right|=o(\sqrt{n})$. Therefore
\begin{equation}
\label{eq:convprob}
\left|\triangle_{\alpha}\right|/\left(2\varepsilon\alpha\sqrt{n\log n}\right) \Longrightarrow 1,
\end{equation}
under $\PP_{id}^n(\Pi_{t_n} \in \cdot)$. From Lemma \ref{lem:delta-U}, the convergence in \eqref{eq:convprob}, clearly,  holds under the stationary measure $\UU^n$ as well.
\end{remark}

\section{\label{sec:pfmix}Proof of Theorem \ref{thm:mixing}}

In order to prove the TV lower bound we consider the deviation of
$\left|\triangle_{\alpha}\right|$ from $2\varepsilon\alpha\sqrt{n\log n}$.
Assume $t_{n}$ is as in the theorem. Let $k>0$ and $\alpha$ be real numbers such that $v\left(\alpha\right)>0$
(where $v\left(\alpha\right)$ was defined in Lemma \ref{lem:delta-A}).
The parameters $k$ and $\alpha$ will be fixed until \eqref{eq:k_alpha_bound}, where we derive a lower bound on the TV distance which depends
on them. Then, maximizing over the two parameters, we shall obtain the required bound on TV distance.

Suppose that for some $n$
\begin{equation}
\label{eq:mixpf1}
\left|\left|\triangle_{\alpha}\left(\Pi_{t}\right)\setminus A^{t_{n}}\right|-2\varepsilon\alpha\left(1-p_{t_{n},n}\right)\sqrt{n\log n}\right|<k\sqrt{6 \varepsilon \alpha } (n\log n)^{1/4},
\end{equation}
and
\begin{equation}
\label{eq:mixpf2}
\left|\triangle_{\alpha}\left(\Pi_{t_{n}}\right)\cap A^{t_{n}}\right|>\frac{1}{2}v\left(\alpha\right)\varepsilon np_{t_{n},n}.
\end{equation}
Then, if $n$ is sufficiently large,
\begin{equation}
  \label{eq:mixpf3}
\begin{aligned}
&\left|\triangle_{\alpha}\left(\Pi_{t_{n}}\right)\right|-2\varepsilon\alpha\sqrt{n\log n} \\
& \quad \geq\varepsilon np_{t_{n},n}\left(\frac{1}{2}v\left(\alpha\right)-2\alpha\sqrt{n\log n}/n\right) -k\sqrt{6 \varepsilon \alpha} (n\log n)^{1/4}\\
 & \quad\geq k \sqrt{2 \varepsilon \alpha} (n\log n)^{1/4},
\end{aligned}
\end{equation}
where the last inequality follows from the following calculation: writing
\[
\log\frac{np_{t_{n},n}}{(n\log n)^{1/4}}=\frac{3}{4}\log n-\frac{1}{4}\log\log n+\log p_{t_{n},n},
\]
substituting $p_{t_{n},n}=(1-1/n)^{t_n}$ and $t_{n}=\frac{3}{4}n\log n-\frac{1}{4}n\log\log n-c_{n}n$, and using the fact that $\log\left(1+x\right)=x+O\left(x^{2}\right)$
as $x\rightarrow0$, we arrive at
\begin{equation}
\label{eq:c_n}
\log\frac{np_{t_{n},n}}{(n\log n)^{1/4}}=c_{n}+o\left(1\right)\rightarrow\infty.
\end{equation}

Now, since for large $n$ \eqref{eq:mixpf1} and \eqref{eq:mixpf2} imply \eqref{eq:mixpf3},  by a union bound, Lemma \ref{lem:delta-A} and Lemma \ref{lem:delta-A^c} imply
\begin{align*}
 & \liminf_{n\rightarrow\infty}\PP_{id}^{n}\left(\left|\triangle_{\alpha}\left(\Pi_{t_{n}}\right)\right|-2\varepsilon\alpha\sqrt{n\log n}\geq
k \sqrt{2 \varepsilon \alpha} (n\log n)^{1/4} \right)\\
 & \quad\geq1-\frac{1}{k^{2}}-\left(1-\frac{1}{2}\right)^{-2}\left(v^{-2}\left(\alpha\right)-1\right) \triangleq \phi(k,\alpha) .
\end{align*}

In addition, from Lemma \ref{lem:delta-U},
\[
\limsup_{n\rightarrow\infty}\UU^{n}\left(\left|\triangle_{\alpha}\left(\sigma\right)\right|-2\varepsilon\alpha\sqrt{n\log n}\geq
k \sqrt{2 \varepsilon \alpha} (n\log n)^{1/4} \right)\leq\frac{1}{k^{2}} .
\]

Thus,
\begin{equation}
\label{eq:k_alpha_bound}
\liminf_{n\rightarrow\infty}\left\Vert \PP_{id}^{n}\left(\Pi_{t_n}\in\cdot\right)-\UU^{n}\right\Vert _{TV} \geq \phi(k,\alpha) - \frac{1}{k^2}.
\end{equation}
Since $k$ and $\alpha$ were arbitrary, and since as $k,\alpha \rightarrow \infty$, $\phi(k,\alpha) \rightarrow 1$ and $\frac{1}{k^2} \rightarrow 0$,
\[
\lim_{n\rightarrow\infty}\left\Vert \PP_{id}^{n}\left(\Pi_{t_n}\in\cdot\right)-\UU^{n}\right\Vert _{TV} = 1.
\]
\qed

\begin{remark}
In Section \ref{sec:delta}, we have seen that the standard deviation of both $\left|\triangle_{\alpha}\right|$
under $\UU^{n}$ and $\left|\triangle_{\alpha}(\Pi_{t_{n}})\setminus A^{t_{n}}\right|$ under $\PP_{id}^{n}$ is of order $\Theta((n\log n)^\frac{1}{4})$.
Since $\EE_{id}^{n}\left\{\left|\triangle_{\alpha}(\Pi_{t_{n}}) \cap A^{t_{n}}\right|\right\}=\Theta(np_{t_{n},n})$, by \eqref{eq:c_n}, it is of higher order than $\Theta((n\log n)^\frac{1}{4})$, if and only if $t_n$  is of the form in Theorem \ref{thm:mixing}.
In particular, this shows why the term $-\frac{1}{4}n \log\log n$ is essential to us in the choice of $t_n$.
\end{remark}
\section{\label{sec:pflem}Proof of Lemma \ref{lem:2ndordProb}}

In this section we prove Lemma \ref{lem:2ndordProb} and additional results needed for the proof.

\begin{proof}
(Lemma \ref{lem:2ndordProb})
For $m_{1}\in\left[n-1\right]$, $m_{2}\in\left[n\right]$ and $0\leq t'\in\mathbb{Z}$,
let $\sigma\in S_{n}$ be some permutation such that $\sigma\left(j\right)=m_{2}$
and
\[
\sigma\left(i\right)=\begin{cases}
m_{1}+1 & \mbox{if }m_{2}\leq m_{1},\\
m_{1} & \mbox{if }m_{2}>m_{1},
\end{cases}
\]
and let $\mathcal{P}_{m_{1},m_{2}}^{t'}=\mathcal{P}_{m_{1},m_{2}}^{n,t'}$
be the probability measure on $\left[n\right]\times\left[n\right]$
defined by
\[
\mathcal{P}_{m_{1},m_{2}}^{n,t'}\left(\cdot\right)=\PP_{\sigma}^{n}\left(\left.\left(\Pi_{t'}\left(i\right),\Pi_{t'}\left(j\right)\right)\in\cdot\,\right|\, i,j\in A^{t'}\right).
\]
(Which, obviously, does not depend on the values $\sigma\left(k\right)$
for $k\notin\left\{ i,j\right\} $.)

That is, starting with a deck whose ordering is obtained by inserting the card numbered
$j$ in position $m_{2}$, in a deck composed of the $n-1$ cards with numbers in $\left[n\right]\setminus\left\{ j\right\}$
in which the position of card $i$ is $m_{i}$, $\mathcal{P}_{m_{1},m_{2}}^{t'}$ is the joint probability law
of the positions of the cards numbered $i$ and $j$ after performing
$t'$ random-to-random insertion shuffles, conditioned on not choosing
either of the cards $i$ and $j$.

Now, let $t$, $t_1$ and $t_2$ be natural numbers as in the statement of the lemma, which will be fixed throughout the proof.
Define the events
\[
Q_{m}^{+}=\left\{ \Pi_{t_{2}-1}\left(i\right)=m,\Pi_{t_{2}-1}\left(j\right)>m\right\} ,
\]
\[
Q_{m}^{-}=\left\{ \Pi_{t_{2}-1}\left(i\right)=m,\Pi_{t_{2}-1}\left(j\right)<m\right\} ,
\]
and define $q_{m}^+$ and $q_{m}^-$ by
\[
q_{m}^{\pm}=\PP_{id}^{n}\left(\left.Q_{m}^{\pm}\,\right|\,\tau_{i}^{t}=t_{1},\tau_{j}^{t}=t_{2}\right).
\]
Define the probability measure $\mu$ on $\left[n\right]\times\left[n\right]$ by
\begin{align}
\mu\left(\cdot\right)\triangleq & \,\PP_{id}^{n}\left(\left.\left(\Pi_{t}\left(i\right),\Pi_{t}\left(j\right)\right)\in\cdot\,\right|\,\tau_{i}^{t}=t_{1},\tau_{j}^{t}=t_{2}\right)\label{eq:P_ij}\\
= & \, \frac{1}{n}\sum_{m_{2}=1}^{n}\left\{ \sum_{m_{1}=1}^{n-1}q_{m_{1}}^{+}\mathcal{P}_{m_{1},m_{2}}^{t-t_{2}}\left(\cdot\right)+\sum_{m_{1}=2}^{n}q_{m_{1}}^{-}\mathcal{P}_{m_{1}-1,m_{2}}^{t-t_{2}}\left(\cdot\right)\right\} .\nonumber
\end{align}

Considering the Markov chain $\Pi_{t'}(i)$, $t'=t_1,\ldots,t_2-1$, conditioned on $\tau^t_i=t_1$ and $\tau^t_j=t_2$,
by an argument similar to that given in the beginning of the proof of Lemma \ref{lem:delta-A^c}, $\Pi_{t_2-1}(i)$ is
uniformly distributed on $[n]$. Thus, for any $m\in\left[n\right]$,
\begin{equation}
q_{m}^{+}+q_{m}^{-}=\frac{1}{n}.\label{eq:q sum}
\end{equation}

Similarly, for any $s \in \mathbb{N}$, the transition matrix of the chain $\left(\Pi_{t'}\left(i\right),\Pi_{t'}\left(j\right)\right)$,
$t'=0,\ldots, s$, conditioned on $i, j \in A^{s}$, is symmetric, and therefore the uniform measure on $\left\{ \left(m_{1},m_{2}\right)\in\left[n\right]^{2}:\, m_{1}\neq m_{2}\right\} $,
which we denote by $\UU_{\left(2\right)}^{n}$, is a stationary measure of the chain (the chain is reducible, thus the
stationary measure is not unique). It therefore follows
that for any $0\leq t'\in\mathbb{Z}$,
\begin{equation}
\frac{1}{n\left(n-1\right)}\sum_{m_{1}=1}^{n-1}\sum_{m_{2}=1}^{n}\mathcal{P}_{m_{1},m_{2}}^{t'}\left(\cdot\right)=\UU_{\left(2\right)}^{n}\left(\cdot\right).\label{eq:uniform_ij}
\end{equation}

Our next step is to define two additional Markov chains $\Pi_{t'}^{-}$
and $\Pi_{t'}^{+}$, $t'=t_{2},t_{2}+1,\ldots,t$, with state space
$S_{n}$, such that on $\left\{ \tau_{i}^{t}=t_{1},\tau_{j}^{t}=t_{2}\right\} $,
\begin{equation}
\Pi_{t'}^{-}\left(i\right)\leq\Pi_{t'}\left(i\right)\leq\Pi_{t'}^{+}\left(i\right)\mbox{\,\,\,\ and\,\,\,\ }\Pi_{t'}^{-}\left(j\right)=\Pi_{t'}\left(j\right)=\Pi_{t'}^{+}\left(j\right),\label{eq:Pi relations}
\end{equation}
for any $t'=t_{2},t_{2}+1,\ldots,t$. Once we have done so, defining $\mu^+$, $\mu^-$ by
\[
\mu^{\pm}\left(\cdot\right)\triangleq \PP_{id}^{n}\left(\left.\left(\Pi_{t}^{\pm}\left(i\right),\Pi_{t}^{\pm}\left(j\right)\right)\in\cdot\,\right|\,\tau_{i}^{t}=t_{1},\tau_{j}^{t}=t_{2}\right),
\]
it will follow that
\begin{equation}
\label{eq:bound_mu}
\begin{aligned}
\mu\left(\overline{i\pm\delta}\times\overline{j\pm\delta}\right)\leq & \,\,\mu\left(\left[n\right]\times\overline{j\pm\delta}\right)\\
 & -\mu^{-}\left(\left(\left(i+\delta,n\right]\cap\left[n\right]\right)\times\overline{j\pm\delta}\right)\\
 & -\mu^{+}\left(\left(\left[1,i-\delta\right)\cap\left[n\right]\right)\times\overline{j\pm\delta}\right).
\end{aligned}
\end{equation}

For each $m_{1}\in\left[n\right]\setminus\left\{ 1,n\right\} $ define
the events $\widehat{Q}_{m_{1}}^+$, $\widehat{Q}_{m_{1}}^-$, $\widehat{Q}^+$ and $\widehat{Q}^-$ by
\begin{equation}
\label{eq:q_m1}
\begin{aligned}
\widehat{Q}_{m_{1}}^{\pm}= & \,Q_{m_{1}}^{\pm}{\textstyle \bigcap}\left\{ \Pi_{t_{2}}\left(j\right)\neq m_{1}\right\} {\textstyle \bigcap}\left\{ \tau_{i}^{t}=t_{1},\tau_{j}^{t}=t_{2}\right\} ,\\
\widehat{Q}^{\pm}= & \underset{m_{1}=2,\ldots,n-1 }{{\textstyle \bigcup}}Q_{m_{1}}^{\pm}.
\end{aligned}
\end{equation}
Let us define $\Pi_{t'}^{+}$
and $\Pi_{t'}^{-}$ by setting, for $t'=t_{2},t_{2}+1,\ldots,t$,
\begin{align*}
\Pi_{t'}^{+}= & \Pi_{t'}^{0}\,\,\,\mbox{on}\,\,\,\widehat{Q}^{-}, & \Pi_{t'}^{+}= & \Pi_{t'}\,\,\,\mbox{on}\,\,\,\left(\widehat{Q}^{-}\right)^{c},\\
\Pi_{t'}^{-}= & \Pi_{t'}^{0}\,\,\,\mbox{on}\,\,\,\widehat{Q}^{+}, & \Pi_{t'}^{-}= & \Pi_{t'}\,\,\,\mbox{on}\,\,\,\left(\widehat{Q}^{+}\right)^{c},
\end{align*}
where $\Pi_{t'}^{0}$ is an additional Markov chain defined on
$\widehat{Q}^{+}\cup\widehat{Q}^{-}$ as described below.

The random walk $\Pi_{t'}$ on $S_{n}$ corresponds to the ordering
of a deck of $n$ cards as it is being shuffled by random-to-random
insertion shuffles. Let us call this deck of cards deck A, and for
each time $t'$ let us denote by $c_{t'}$ and $d_{t'}$ the number
of the card removed from the deck at that time and the position into which
it is reinserted, respectively. (To avoid any confusion --
we refer to the ordering of the deck after the removal the card numbered
$c_{t'}$ and its reinsertion to position $d_{t'}$ as the state of
the deck at time $t'$, and not $t'+1$.)

In order to define $\Pi_{t'}^{0}$, we describe a shuffling process
on a deck of $n$ cards, which we shall refer to as deck B, on the
set of times $t'=t_{2},t_{2}+1,\ldots,t$, and set $\Pi_{t'}^{0}$
to be the permutation corresponding to the ordering of the deck at
time $t'$ (i.e., $\Pi_{t'}^{0}\left(k\right)$ is the position of
the card numbered $k$).

We begin by defining the state of deck B at time $t'=t_{2}$
on $\widehat{Q}^{+}$ (respectively, $\widehat{Q}^{-}$) as the deck
obtained by taking a deck of $n$ cards ordered as deck A is ordered at
the same time and transposing the card numbered $i$ with the card which
has position lower (receptively, higher) by $1$ from the card numbered
$i$.

For a given state of decks A and B, let us say that two cards with numbers
in $\left[n\right]\setminus\left\{ i,j\right\} $, one in each of
the decks, `match' each other, if after removing the cards
numbered $i$ and $j$ from both decks they have the same position.

At each of the times $t'=t_{2}+1,\ldots,t$, suppose
deck B is shuffled based upon how deck A is as follows: when the card
numbered $c_{t'}$ is removed from deck A, we also remove the matching
card from deck B; then, we reinsert both cards to their decks in the
same position, $d_{t'}$. This defines the state of deck B, and thus $\Pi^0_{t'}$, for times
$t'=t_{2}+1,\ldots,t$.

A concrete example of a simultaneous shuffle of both decks with $n=8$,
$i=5$ and $j=7$ is given in Figure \ref{fig:deckAB}. Cards $i$
and $j$ are colored gray. Card $1$ is chosen for
removal in deck A, and so the matching card, $4$, is the one removed
from deck B. Then, they are reinserted in the same position.
\begin{figure}[h!]
\hspace{-1.5em}\includegraphics[width=1\textwidth]{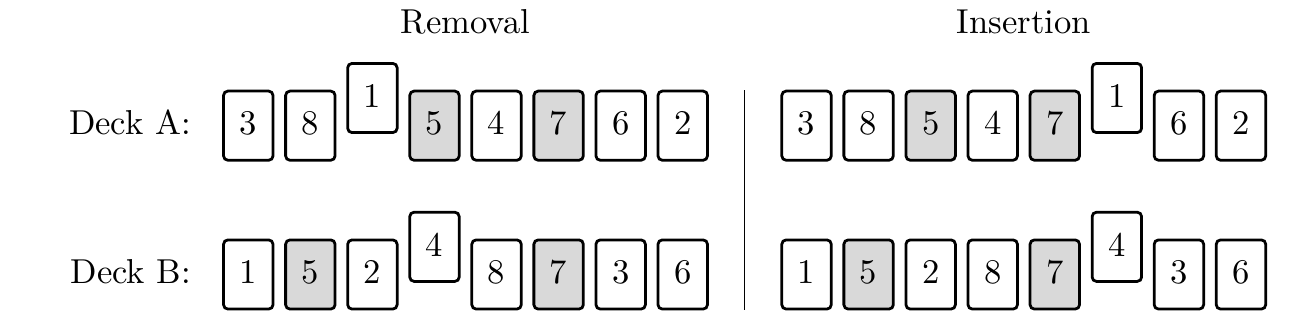}\\
\caption{\label{fig:deckAB}A shuffle of deck A and deck B.}
\end{figure}

Directly from definition, \eqref{eq:Pi relations} holds for $t'=t_{2}$.
It is also easy to verify that every single shuffle of decks A and
B as described above preserves the relations in \eqref{eq:Pi relations}, which implies that,
indeed, (\ref{eq:Pi relations}) holds for any $t'=t_{2},t_{2}+1,\ldots,t$.

Note that, by definition, 
\begin{align*}
\mu^{+}\left(\cdot\right) & = \sum_{m=1}^{n-1} q^+_m \PP_{id}^{n}\left(\left.\left(\Pi_{t}^{+}\left(i\right),\Pi_{t}^{+}\left(j\right)\right)\in\cdot\,\right|\, Q^+_m,\tau_{i}^{t}=t_{1},\tau_{j}^{t}=t_{2}\right)\\
& + \sum_{m=2}^{n} q^-_m \PP_{id}^{n}\left(\left.\left(\Pi_{t}^{+}\left(i\right),\Pi_{t}^{+}\left(j\right)\right)\in\cdot\,\right|\, Q^-_m,\tau_{i}^{t}=t_{1},\tau_{j}^{t}=t_{2}\right)\\
& = \frac{1}{n} \sum_{m_1=1}^{n-1} \sum_{m_2=1}^{n} q^+_{m_1} \mathcal{P}_{m_1,m_2}^{t-t_{2}}\left(\cdot\right)
+\frac{1}{n} \sum_{m_1=2}^{n} q^-_{m_1} \mathcal{P}_{m_1-1,m_1}^{t-t_{2}}\left(\cdot\right)\\
& + \frac{1}{n} \sum_{m_1=2}^{n} \sum_{m_2\in [n]\setminus \{m_1\}} q^-_{m_1} \mathcal{P}_{m_1,m_2}^{t-t_{2}}\left(\cdot\right).
\end{align*}
From this, together with (\ref{eq:q sum}) and (\ref{eq:uniform_ij}),
we obtain
\begin{equation}
\label{eq:mu+}
\begin{aligned}
\mu^{+}\left(\cdot\right) & =\frac{n-1}{n}\UU_{\left(2\right)}^{n}\left(\cdot\right)+\frac{1}{n^{2}}\sum_{m=1}^{n}\mathcal{P}_{n-1,m}^{t-t_{2}}\left(\cdot\right)\\
 & \,\,+\frac{1}{n}\sum_{m=2}^{n-1}\left\{ q_{m}^{-}\mathcal{P}_{m-1,m}^{t-t_{2}}\left(\cdot\right)-q_{m}^{-}\mathcal{P}_{m,m}^{t-t_{2}}\left(\cdot\right)\right\}.
\end{aligned}
\end{equation}
Similarly,
\begin{equation}
\label{eq:mu-}
\begin{aligned}
\mu^{-}\left(\cdot\right) & =\frac{n-1}{n}\UU_{\left(2\right)}^{n}\left(\cdot\right)+\frac{1}{n^{2}}\sum_{m=1}^{n}\mathcal{P}_{1,m}^{t-t_{2}}\left(\cdot\right)\\
 & \,\,+\frac{1}{n}\sum_{m=2}^{n-1}\left\{ q_{m}^{+}\mathcal{P}_{m,m}^{t-t_{2}}\left(\cdot\right)-q_{m}^{+}\mathcal{P}_{m-1,m}^{t-t_{2}}\left(\cdot\right)\right\} .
\end{aligned}
\end{equation}

According to (\ref{eq:Pi relations}), $\mu\left(\left[n\right]\times\overline{j\pm\delta}\right)=\mu^{+}\left(\left[n\right]\times\overline{j\pm\delta}\right)$.
Hence, by substitution of (\ref{eq:mu+}) and (\ref{eq:mu-}) in (\ref{eq:bound_mu}),
and using (\ref{eq:q sum}), it can be easily shown that
\begin{equation}
\label{eq:mu_fin_approx}
\begin{aligned}
& \mu\left(\overline{i\pm\delta}\times\overline{j\pm\delta}\right)\leq \frac{n-1}{n}\UU_{\left(2\right)}^{n}\left(\overline{i\pm\delta}\times\overline{j\pm\delta}\right)+\\ & \qquad \frac{1}{n^{2}}\sum_{m=1}^{n}\mathcal{P}_{n-1,m}^{t-t_{2}}\left(\left[n\right]\times\overline{j\pm\delta}\right) +\frac{1}{n^{2}}\sum_{m=2}^{n-1}\mathcal{P}_{m-1,m}^{t-t_{2}}\left(\left[n\right]\times\overline{j\pm\delta}\right).
\end{aligned}
\end{equation}

The first summand is bounded by
\[
\frac{n-1}{n}\UU_{\left(2\right)}^{n}\left(\overline{i\pm\delta}\times\overline{j\pm\delta}\right)\leq\frac{\left(1+2\left\lfloor \delta\right\rfloor \right)^{2}}{n^{2}}.
\]

Note that, for fixed $m_{2}$, $\mathcal{P}_{m_{1},m_{2}}^{t_{2}-t'}\left(\left[n\right]\times\overline{j\pm\delta}\right)$
is identical for all $m_{1}$ such that $m_{1}<m_{2}$, and for all
$m_{1}$ such that $m_{1}\geq m_{2}$. Thus,
\begin{equation}
\nonumber
\sum_{m=2}^{n-1}\mathcal{P}_{m-1,m}^{t-t_{2}}\left(\left[n\right]\times\overline{j\pm\delta}\right)=\sum_{m=2}^{n-1}\mathcal{P}_{1,m}^{t-t_{2}}\left(\left[n\right]\times\overline{j\pm\delta}\right).\label{eq:s1}
\end{equation}
Corollary \ref{cor:sum bound} below provides an upper bound for this
sum. Bounding the additional sum in (\ref{eq:mu_fin_approx}) by the
same bound can be done similarly, which completes the proof.
\end{proof}
Corollary \ref{cor:sum bound}, used in the previous proof, will follow
from the following.
\begin{lem}
\label{lem:coupling ijm}For any real number $\delta\geq0$,
integers $r,\,t\geq0$, $i,j\in\left[n\right]$, and $m\in\left[n-1\right]$,
\begin{align*}
\mathcal{P}_{1,m+1}^{n,t}\left(\left[n\right]\times\overline{j\pm\delta}\right) & \leq \PP_{id}^{n-1}\left(\left.\Pi_{t}\left(1\right)>r\,\right|\,1\in A^{t}\right)+\\
 & \,\, \PP_{id}^{n-1}\left(\left.\Pi_{t}\left(m\right)\in\overline{j\pm\left(\delta+r\right)}\,\right|\, m\in A^{t}\right).
\end{align*}

\end{lem}
Before we turn to proof of the lemma, let us state and prove the above mentioned corollary.
\begin{cor}
\label{cor:sum bound}For any real number $\delta\geq0$, integer
$0\leq t\leq n\log n$, and $j\in\left[n\right]$,
\[
\sum_{m=2}^{n-1}\mathcal{P}_{1,m}^{n,t}\left(\left[n\right]\times\overline{j\pm\delta}\right)\leq2\delta+\widehat{g}\left(n\right),
\]
where $\widehat{g}\left(n\right)=\Theta\left(\log^{2}n\right)$ is
a function independent of the parameters above.
\end{cor}
\begin{proof}
From Lemma \ref{lem:coupling ijm}, for any real $\delta \geq0$,
integers $r,\,t\geq0$, and $j\in\left[n\right]$,
\begin{equation}
\label{eq:cor1}
\begin{aligned}
\sum_{m=2}^{n-1}\mathcal{P}_{1,m}^{n,t}\left(\left[n\right]\times\overline{j\pm\delta}\right) & \leq\left(n-2\right)\PP_{id}^{n-1}\left(\left.\Pi_{t}\left(1\right)>r\,\right|\,1\in A^{t}\right)+\\
 & \,\,\sum_{m=1}^{n-1}\PP_{id}^{n-1}\left(\left.\Pi_{t}\left(m\right)\in\overline{j\pm\left(\delta+r\right)}\,\right|\, m\in A^{t}\right).
\end{aligned}
\end{equation}

Clearly, the transition probabilities of $\Pi_{t'}\left(m\right)$, $t'=0,1,\ldots,t$,
conditioned on $m\in A^{t}$ do not depend on $m$. Thus, up to a factor of $n-1$, the sum on the
right-hand side above is equal to the probability that at time $t$, the state of the
 Markov chain with those transition probabilities and with uniform initial distribution belongs
to $\overline{j\pm\left(\delta+r\right)}$.
Since the transition matrix of this chain is symmetric, the stationary measure for
this chain is the uniform measure. Thus,
\begin{equation}
\label{eq:cor2}
\sum_{m=1}^{n-1}\PP_{id}^{n-1}\left(\left.\Pi_{t}\left(m\right)\in\overline{j\pm\left(\delta+r\right)}\,\right|\, m\in A^{t}\right)\leq1+2\left(\delta+r\right).
\end{equation}

By \eqref{eq:31} and by the same argument as in
\eqref{eq:zetamax}, setting $t_{n}=\left\lfloor n\log n\right\rfloor $, for
any sequence of integers $r_{n}\geq0$,
\begin{align*}
&\PP_{id}^{n-1}\left( \left.  \Pi_t\left(1\right) >r_{n}\,\right|\,1\in A^{t_{n}}\right) \leq
 \PP_{id}^{n-1}\left( \left. \max_{0\leq t'\leq t_{n}} \left|\Pi_{t_{n}}\left(1\right) -1 \right| \geq r_{n}\,\right|\,1\in A^{t_{n}}\right) \\
 & \qquad = \PP\left(\max_{0\leq t'\leq t_{n}}\left|\zeta_{t'}^{n-1,1,r_{n}}-1\right|\geq r_{n}\right)\leq4\PP\left(S_{t_{n}}^{n-1,1,r_{n}}\geq r_{n}/3\right)\\
 & \qquad\quad+4\PP\left(W_{t_{n}}^{n-1,r_{n}}\geq r_{n}/3\right)+\PP\left(N_{t_{n}}^{n-1}\geq r_{n}/3\right).
\end{align*}
Using Bernstein inequalities (\cite{pPET95a}, Theorem 2.8), it is easy to verify that one
can choose a sequence $r_{n}=\Theta\left(\log^{2}n\right)$ such that the last part of the inequality above is
$o\left(\log^{2}n/n\right)$. From this, together with \eqref{eq:cor1} and \eqref{eq:cor2}, the corollary follows.
\end{proof}
We now turn the proof of Lemma \ref{lem:coupling ijm}.
\begin{proof}
(Lemma \ref{lem:coupling ijm}) The proof is based on a coupling
of the Markov chains corresponding to the shuffling of two decks of cards.
The first of the two decks contains $n$ cards, numbered from
$1$ to $n$, and at time $0$ (the initial state) has card $i$ at
position $1$ and card $j$ at position $m+1$. The second deck contains
$n-1$ cards, numbered from $1$ to $n-1$, and at time $0$ is ordered
lexicographically, i.e., according to the numbers of the cards. Let us call
the decks deck 1 and deck 2, respectively.

We want to define a procedure to simultaneously shuffle the decks such that:
\begin{enumerate}
\item At each step deck 1 is shuffled by choosing a random card, different from $i$ and $j$, removing it from the deck,
and inserting it back into the deck at a random position; with shuffles at different steps being independent.
\item At each step deck 2 is shuffled by choosing a random card, different from $m$, removing it from the deck,
and inserting it back into the deck at a random position; with shuffles at different steps being independent.
\item For all $t\geq0$,
\begin{equation}
J_{t}-I_{t}\leq M_{t}\leq J_{t}-1,\label{eq:Is&Js}
\end{equation}
where $J_t$ (respectively, $I_t$) denotes the position of the card numbered $j$ (respectively,
$i$) in deck 1 after completing $t$ shuffles, and $M_t$ denotes the position of card $m$ in deck 2 after completing $t$ shuffles.
\end{enumerate}

We shall also need the notation $\bar{J}_t$ (respectively, $\bar{I}_t$) for the position
of the card numbered $j$ (respectively, $i$) in deck 1, after completing
$t-1$ shuffles and performing only the removal of the $t$-th shuffle.
Note that since after the removal the deck contains only $n-1$ cards,
these positions are values in $\left[n-1\right]$. Similarly, $\bar{M}_t$
shall denote the corresponding position of the card numbered $m$ in deck 2.

The definition of the shuffling shall be done inductively, and so, let us begin by assuming that (\ref{eq:Is&Js})
holds for some time $t'$. Under the assumption, one can easily define
a bijection from the set of cards in deck 1 that are different from
$i$ and $j$, to the set cards in deck 2 that are different from
$m$, such that at time $t'$:
\begin{enumerate}
\item any card with position between the
cards numbered $i$ and $j$ in\\ deck 1 is mapped to a card below $m$
in deck 2; and
\item any card below $m$ in deck 2 is the image of some
card below $j$ \\in deck 1.
\end{enumerate}
See, for example, Figure \ref{fig:deckMatchings}.

\begin{figure}[h!]
\hspace{-2em}\includegraphics[width=1\textwidth]{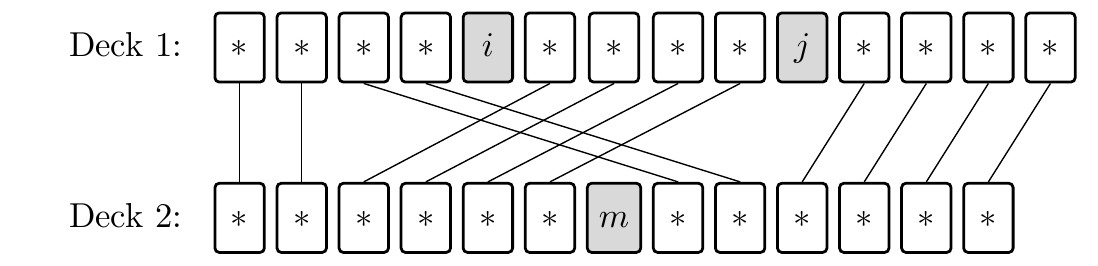}\\
\caption{\label{fig:deckMatchings}A bijection for decks 1 and 2.}
\end{figure}

Once the bijection is defined, one can perform the removal of step $t'+1$ from
both decks by choosing a random card (different from $i$ and $j$)
from deck 1 and removing this card from deck 1 and its image under
the bijection from deck 2. This ensures that
\begin{equation}
\bar{J}_{t'+1}-\bar{I}_{t'+1}\leq\bar{M}_{t'+1}\leq\bar{J}_{t'+1}-1.\label{eq:IJMbar}
\end{equation}

Denote
\begin{align*}
V_{1} & \triangleq\left(J_{t'+1}-I_{t'+1}\right)-\left(\bar{J}_{t'+1}-\bar{I}_{t'+1}\right),\\
V_{2} & \triangleq M_{t'+1}-\bar{M}_{t'+1}\,\,,\,\, V_{3}\triangleq J_{t'+1}-\bar{J}_{t'+1},
\end{align*}
and note that $V_{1},\, V_{2},\, V_{3}\in\left\{ 0,1\right\} $.

If we assume that the first two of the three
conditions we need the shuffling to satisfy hold, then the conditional probabilities
\begin{equation}
\nonumber
p_i = p_i(\bar{I}_{t'+1},\bar{J}_{t'+1},\bar{M}_{t'+1}) \triangleq \PP(V_i=1|\bar{I}_{t'+1},\bar{J}_{t'+1},\bar{M}_{t'+1}) \, , \quad i=1,\,2,\,3,
\end{equation}
satisfy
\[
p_{1}=\frac{\bar{J}_{t'+1}-\bar{I}_{t'+1}}{n}\leq p_{2}=\frac{\bar{M}_{t'+1}}{n-1}\leq p_{3}=\frac{\bar{J}_{t'+1}}{n}.
\]

Therefore, since $\left\{ V_{1}=1\right\} \subset\left\{ V_{3}=1\right\} $,
it is possible to couple the reinsertions of the cards back
to their decks at step $t'+1$, so that the position of each of the cards after reinsertion
is uniform in its deck, and so that (\ref{eq:Is&Js}) also holds for
time $t'+1$.

By induction, this completes our definition of the shuffling of the two decks and implies that for any integers $t,\,r\geq0$,
\begin{align*}
\mathcal{P}_{1,m+1}^{n,t}\left(\left[n\right]\times\overline{j\pm\delta}\right) & \leq\mathcal{P}_{1,m+1}^{n,t}\left(\left(\left[n\right]\setminus\left[1,r\right]\right)\times\left[n\right]\right)+\\
 & \,\, \PP_{id}^{n-1}\left(\left.\Pi_{t}\left(m\right)\in\overline{j\pm\left(\delta+r\right)}\,\right|\, m\in A^{t}\right).
\end{align*}

To finish the proof, note that by a coupling argument (remove cards
as described in the proof of Lemma \ref{lem:2ndordProb} for decks
A and B, with the difference of removing card $m$ instead of cards
$i$ and $j$ from the smaller deck in order to compare positions for `matchings', and define the
random insertion appropriately),
\[
\mathcal{P}_{1,m+1}^{n,t}\left(\left(\left[n\right]\setminus\left[1,r\right]\right)\times\left[n\right]\right)\leq \PP_{id}^{n-1}\left(\left.\Pi_{t}\left(1\right)>r\,\right|\,1\in A^{t}\right).
\]
\end{proof}

\section*{Acknowledgements}
This work arose from a graduate course taught by Professor Ross G. Pinsky. I am grateful to him for introducing the problem
and for valuable comments throughout the preparation of this work. I would also like to thank my advisor, Professor Robert J. Adler,
his helpful remarks on earlier versions of this paper.

\bibliographystyle{amsplain}
\bibliography{master}

\end{document}